\documentclass[preprint,12pt]{elsarticle}

\usepackage{amscd}
\usepackage{amsmath}
\usepackage{amsthm}
\usepackage{amsfonts}
\usepackage{graphicx}
\usepackage{amssymb}
\usepackage{color}
\usepackage{amsbsy}
\usepackage{graphicx}
\usepackage{amssymb,color,amsbsy}
\usepackage{mathtools}
\usepackage{leftindex}

\usepackage{float}
\usepackage{comment}

\usepackage{stmaryrd}

\usepackage[ruled]{algorithm2e}

\usepackage[matrix,arrow]{xy}

\DeclareMathAlphabet{\mathpzc}{OT1}{pzc}{m}{it}

\usepackage{pgfpages}
\usepackage{tikz}
\usetikzlibrary{shapes.geometric}
\usetikzlibrary{decorations.pathmorphing}
\usetikzlibrary{arrows}
\usetikzlibrary{backgrounds}
\usetikzlibrary{positioning}
\usetikzlibrary{fit}
\usepackage{caption}
\usepackage{amsthm}

\usepackage{amsmath}

\usepackage[pagebackref=true, bookmarksopen=true,colorlinks=true, linkcolor=red,citecolor=blue]{hyperref}

\global\long\def\x{\mathcal{\times}}%

\global\long\def\j{\mathcal{J}}%

\global\long\def\ñ{\sim}%

\usepackage{tikz-cd}

\usepackage{tkz-graph}

\usepackage{multicol}

\usepackage{subcaption}

\newtheorem{theorem}{Theorem}[section]

\newtheorem{corollary}[theorem]{Corollary}

\newtheorem{example}[theorem]{Example}

\newtheorem{proposition}[theorem]{Proposition}

\newtheorem{question}[theorem]{Question}

\usepackage[capitalise]{cleveref}  
\crefname{theorem}{Theorem}{Theorems}
\Crefname{theorem}{Theorem}{Theorems}
\crefname{corollary}{Corollary}{Corollaries}
\Crefname{corollary}{Corollary}{Corollaries}
\crefname{proposition}{Proposition}{Propositions}
\Crefname{proposition}{Proposition}{Propositions}
\crefname{example}{Example}{Examples}
\Crefname{example}{Example}{Examples}
\crefname{question}{Question}{Questions}
\Crefname{question}{Question}{Questions}

\newcommand{\lk}{\operatorname{lk}}
\newcommand{\cl}{\operatorname{cl}}
\newcommand{\codeg}{\operatorname{codeg}}
\newcommand{\sat}{\operatorname{sat}}

\newcommand{\widx}{\operatorname{w}}

\usepackage{comment}

\begin{document}
\begin{frontmatter} 
\title{A Ridge--Saturation Characterization of \texorpdfstring{$\alpha$-Critical $\mathbf W_p$}{alpha-Critical Wp} Graphs}

\author[Hoang]{Do Trong Hoang}\ead{hoang.dotrong@hust.edu.vn}

\author[LEVIT]{Vadim E. Levit}  \ead{levitv@ariel.ac.il}

\author[MANDRESCU]{Eugen Mandrescu}  \ead{eugen_m@hit.ac.il}
 
\author[DEPTO]{Kevin Pereyra}  \ead{kdpereyra@unsl.edu.ar, kevin.pereyra767@gmail.com}

 \address[Hoang]{Faculty of Mathematics and Informatics, 
 Hanoi University of Science and Technology, 
 1 Dai Co Viet, Bach Mai, Hanoi, Vietnam}
 
  \address[LEVIT]{Department of Mathematics, Ariel University, Ariel, Israel.}

  \address[MANDRESCU]{Department of Computer Science, Holon Institute of Technology, Holon, Israel.}

 
 \address[DEPTO]{Departamento de Matem\'atica, Universidad Nacional de San Luis, San Luis, Argentina.}  

\begin{abstract}
 We characterize the graphs which are simultaneously $\alpha$-critical and members of the class $\mathbf W_p$.  The characterization is stated in three equivalent languages.  In the graph itself, such a graph is a well-covered graph whose codimension-one localization fibers all have size at least $p$ and whose edges are exactly covered by the cliques induced by those fibers.  In the independence complex, it is a pure flag complex in which every ridge has degree at least $p$ and every missing edge is generated by the link of a ridge.  In the complement, it is a $K_{r+1}$-saturated graph, where $r=\alpha(G)$, all maximal cliques have size $r$, and the minimum $(r-1)$-clique-codegree is at least $p$.  This gives an exact formula for the largest $p$ for which a well-covered graph belongs to $\mathbf W_p$.  We make this complement correspondence explicit, record saturation-theoretic consequences including dense-complement rigidity and $p$-sensitive edge and order bounds, and give a family of sharp examples showing that the local sufficient condition from the recent work of Hoang, Levit and Mandrescu is not necessary outside the locally triangle-free setting, for all $p\ge2$.
 \end{abstract}


\end{frontmatter} %
\section{Introduction}
 
 All graphs in this paper are finite, simple, and non-empty.  We write $\Omega^*(G)$ for the family of all independent sets of $G$ and $\Omega(G)$ for the family of maximum independent sets.  A graph is \emph{well-covered} if all maximal independent sets have the same cardinality.  This notion was introduced by Plummer and has since become a central object in the structure theory of independence in graphs \cite{Plummer1970,Plummer1993}.  General recognition and structural classification of well-covered graphs is difficult; useful positive results are known only under additional restrictions or in special families, including girth constraints, planar families, product constructions, and algebraic subclasses \cite{FinbowHartnellNowakowski1993,Pinter1995,ToppVolkmann1992,CastrillonCruzReyes2016,Woodroofe2009}.
 
 Staples introduced the classes $\mathbf W_p$: a graph $G$ belongs to $\mathbf W_p$ if every $p$ pairwise disjoint sets in $\Omega^*(G)$ can be extended to $p$ pairwise disjoint sets in $\Omega(G)$ \cite{Staples1975,Staples1979}.  Thus $\mathbf W_1$ is the class of well-covered graphs, and the classes form the descending chain
 \[
 \mathbf W_1 \supseteq \mathbf W_2 \supseteq \cdots \supseteq \mathbf W_p \supseteq \cdots .
 \]
 The subclass $\mathbf W_2$ is the class of $1$-well-covered graphs without isolated vertices; it was studied by Staples, Pinter, and later by Levit and Mandrescu through shedding vertices, localizations, and differential conditions \cite{Staples1979,Pinter1991,Pinter1992,LevitMandrescu2017,LevitMandrescu2019,BermudoFernau2012}.  Further constructions and polynomial properties of $\mathbf W_p$ graphs appear in \cite{Favaron1982,HoangLevitMandrescuPhamSSRN,HoangLevitMandrescuPhamLog}.  Connections with Cohen--Macaulayness, Gorenstein graphs, Buchsbaum properties, and edge ideals are developed in \cite{HoangTrung2016,HoangTrung2018,JaramilloVillarreal2021,CastrillonCruzReyes2016,Woodroofe2009}.
 
 An edge $e$ of $G$ is $\alpha$-critical if deleting $e$ increases the independence number.  The graph is $\alpha$-critical if every edge is $\alpha$-critical.  This is the independence formulation of the classical line-critical graphs studied by Erd\H{o}s and Gallai, Beineke--Harary--Plummer, Berge, and Plummer \cite{ErdosGallai1961,BeinekeHararyPlummer1967,Berge1982,Plummer1967}.  Plummer asked for a characterization of graphs which are both $\alpha$-critical and well-covered, and in particular for the class $\mathbf W_2$ \cite{Plummer1993}.  A recent paper of Hoang, Levit and Mandrescu developed new characterizations of $\mathbf W_p$ graphs and gave a characterization under a locally triangle-free hypothesis \cite{HoangLevitMandrescu2025}.  The locally triangle-free part relies on earlier triangle-free and algebraic results of Hoang and Trung \cite{HoangTrung2016,HoangTrung2018}.
 
 The purpose of this note is to give an exact characterization for all $p$.  The key point is that the $\mathbf W_p$ condition is not best viewed vertex by vertex.  It is a ridge-thickness condition on the independence complex.  Let $r=\alpha(G)$, and let $S$ be an independent set of size $r-1$.  Define the \emph{fiber over $S$}
 \[
 F_G(S)=\{x\in V(G)\setminus S: S\cup\{x\}\text{ is a maximum independent set of }G\}.
 \]
 Equivalently, $F_G(S)$ is the vertex set of the localization $G_S=G-N_G[S]$.  If $G$ is well-covered, every such $F_G(S)$ is a clique of $G$.  The main theorem says that the desired graphs are exactly those for which these fiber cliques are large and cover all edges.
 
 \begin{theorem}[Main characterization]\label{thm:main-intro}
 Let $p\ge 1$, let $G$ be a graph, and put $r=\alpha(G)$.  The following are equivalent.
 \begin{enumerate}
 \item[(a)] $G$ is $\alpha$-critical and $G\in \mathbf W_p$.
 \item[(b)] $G$ is well-covered, $|F_G(S)|\ge p$ for every independent set $S$ with $|S|=r-1$, and
 \[
 E(G)=\bigcup_{\substack{S\in\Omega^*(G)\\ |S|=r-1}} \binom{F_G(S)}{2}.
 \]
 \item[(c)] If $H=\overline G$, then $H$ is $K_{r+1}$-saturated, every maximal clique of $H$ has size $r$, and every $(r-1)$-clique of $H$ is contained in at least $p$ copies of $K_r$.
 \end{enumerate}
 \end{theorem}
 
 The complement formulation connects the problem with graph saturation, initiated by Erd\H{o}s, Hajnal and Moon \cite{ErdosHajnalMoon1964}.  It separates the two mechanisms: $\alpha$-criticality becomes $K_{r+1}$-saturation, whereas membership in $\mathbf W_p$ becomes a lower bound on the clique-codegree of $(r-1)$-cliques.  This correspondence also yields complement-degree and small-order rigidity consequences and isolates a refined edge-localization problem.

The final examples emphasize why this formulation is different from the edge-localization sufficient condition of \cite{HoangLevitMandrescu2025}.  In particular, for every $q\ge2$ we include a connected graph in $\mathbf W_q$ which is $\alpha$-critical, but for which $G_{ab}$ is not even well-covered for a suitable edge $ab$.

 \section{Preliminaries}
 
 We fix the notation used in the rest of the paper.
 
 Let $G$ be a finite simple graph.  Its vertex and edge sets are $V(G)$ and $E(G)$, its order is $n(G)=|V(G)|$, and its complement is $\overline G$.  The complete graph on $t$ vertices is denoted by $K_t$.  A $t$-clique of a graph $G$ is a vertex set inducing a copy of $K_t$, and
 \[
 \mathcal K_t(G)=\{Q\subseteq V(G): G[Q]\cong K_t\}
 \]
 denotes the family of all $t$-cliques of $G$.  The clique number of $G$ is $\omega(G)=\max\{t:\mathcal K_t(G)\neq\emptyset\}$.  The minimum degree of $G$ is $\delta(G)$.
 
 For $t\ge2$, a graph $G$ is \emph{$K_t$-free} if $\mathcal K_t(G)=\emptyset$.  It is \emph{$K_t$-saturated} if it is $K_t$-free and, for every non-edge $xy$ of $G$, the graph $G+xy$ contains a copy of $K_t$.  Equivalently, $G$ is a maximal $K_t$-free graph on its fixed vertex set.  This is the standard saturation notion introduced by Erd\H{o}s, Hajnal and Moon \cite{ErdosHajnalMoon1964}.  A graph is $r$-partite if its vertex set can be partitioned into $r$ independent sets.  It is complete $r$-partite if every two vertices in different parts are adjacent.
 
 An \emph{independent set} of $G$ is a set of pairwise non-adjacent vertices.  We use
 \[
 \Omega^*(G)=\{S\subseteq V(G):S\text{ is independent in }G\}
 \]
 for the family of all independent sets of $G$.  The independence number is
 \[
 \alpha(G)=\max\{|S|:S\in\Omega^*(G)\},
 \]
 and
 \[
 \Omega(G)=\{S\in\Omega^*(G):|S|=\alpha(G)\}
 \]
 is the family of maximum independent sets.  An independent set is maximal if it is inclusion-maximal in $\Omega^*(G)$; it need not belong to $\Omega(G)$.  The graph $G$ is \emph{well-covered} if every maximal independent set belongs to $\Omega(G)$, equivalently, if all maximal independent sets have cardinality $\alpha(G)$ \cite{Plummer1970,Plummer1993}.
 
 For a positive integer $p$, a graph $G$ belongs to $\mathbf W_p$ if $n(G)\ge p$ and, whenever $A_1,\ldots,A_p\in\Omega^*(G)$ are pairwise disjoint, there exist pairwise disjoint sets $M_1,\ldots,M_p\in\Omega(G)$ such that $A_i\subseteq M_i$ for all $i$.  Thus $\mathbf W_1$ is the class of well-covered graphs, and
 \[
 \mathbf W_1\supseteq \mathbf W_2\supseteq\cdots\supseteq\mathbf W_p\supseteq\cdots .
 \]
 The classes $\mathbf W_p$ were introduced by Staples \cite{Staples1975,Staples1979}.
 
 For $S\subseteq V(G)$, set
\begin{eqnarray*}
 N_G(S)&=&\{v\in V(G)\setminus S: v\text{ is adjacent to some vertex of }S\},
\\
N_G[S]&=&S\cup N_G(S),  
\end{eqnarray*} 
 and define the \emph{localization} of $G$ at $S$ by
 \[
 G_S=G-N_G[S].
 \]
 For $S=\{x\}$ we write $G_x$.  For an edge $xy\in E(G)$, we write
 \[
 G_{xy}=G-(N_G(x)\cup N_G(y));
 \]
 since $xy$ is an edge, this agrees with $G_{\{x,y\}}$.
 
 If $e\in E(G)$, then $G-e$ denotes the graph obtained by deleting $e$ and keeping the same vertex set.  The edge $e$ is \emph{$\alpha$-critical} if $\alpha(G-e)>\alpha(G)$; since deleting one edge can increase $\alpha$ by at most one, this is equivalent to $\alpha(G-e)=\alpha(G)+1$.  A graph is \emph{$\alpha$-critical} if every edge is $\alpha$-critical.  This notion goes back to the line-critical graph literature of Erd\H{o}s--Gallai, Beineke--Harary--Plummer, Berge, and Plummer \cite{ErdosGallai1961,BeinekeHararyPlummer1967,Berge1982,Plummer1967}.
 
 A \emph{simplicial complex} $\Delta$ on a finite vertex set $X$ is a family of subsets of $X$ which contains the empty set and is closed under taking subsets.  The elements of $\Delta$ are its \emph{faces}.  The vertices of $\Delta$ are the one-element faces, and we write $V(\Delta)$ for their set.  The \emph{dimension} of a face $F\in\Delta$ is
 \[
 \dim F=|F|-1,
 \]
 with the convention $\dim\emptyset=-1$.  The \emph{dimension} of $\Delta$ is
 \[
 \dim\Delta=\max\{\dim F:F\in\Delta\};
 \]
 equivalently, $\dim\Delta=d$ means that the largest faces of $\Delta$ have cardinality $d+1$.  A \emph{facet} is an inclusion-maximal face.  The complex $\Delta$ is \emph{pure} if all its facets have dimension $\dim\Delta$, equivalently, if all its facets have the same cardinality.
 
 A subset $A\subseteq V(\Delta)$ is a \emph{non-face} if $A\notin\Delta$.  It is a \emph{minimal non-face} if $A\notin\Delta$ but every proper subset of $A$ is a face.  A \emph{missing edge} of $\Delta$ is a minimal non-face of cardinality two.  The \emph{$1$-skeleton} of $\Delta$ is the graph on $V(\Delta)$ whose edges are the two-element faces of $\Delta$.  The complex $\Delta$ is \emph{flag} if every clique of its $1$-skeleton is a face of $\Delta$.  Equivalently, $\Delta$ is flag if all its minimal non-faces have cardinality two; in this case the whole complex is determined by its $1$-skeleton.
 
 The \emph{independence complex} of $G$ is the simplicial complex
 \[
 \Delta(G)=\Omega^*(G),
 \]
 whose faces are the independent sets of $G$.  Hence $\dim\Delta(G)=\alpha(G)-1$.  Moreover, $G$ is well-covered with $\alpha(G)=r$ if and only if $\Delta(G)$ is pure of dimension $r-1$.  The complex $\Delta(G)$ is flag: a set of vertices fails to be independent precisely when it contains an edge of $G$, so its minimal non-faces are exactly the two-element sets $\{x,y\}$ with $xy\in E(G)$.  Thus the missing edges of $\Delta(G)$ are exactly the edges of $G$, and the $1$-skeleton of $\Delta(G)$ is the complement graph $\overline G$.
 
 For a face $S$ of a simplicial complex $\Delta$, the \emph{link} of $S$ in $\Delta$ is
 \[
 \lk_\Delta(S)=\{T\in\Delta:T\cap S=\emptyset\text{ and }S\cup T\in\Delta\}.
 \]
 If $\Delta$ is pure of dimension $r-1$, a \emph{ridge} of $\Delta$ is a face of dimension $r-2$, equivalently a face of cardinality $r-1$.  The \emph{ridge degree} of a ridge $S$ is the number of vertices in its link:
 \[
 \deg_\Delta(S)=|V(\lk_\Delta(S))|.
 \]
 When $\Delta=\Delta(G)$ and $S\in\Omega^*(G)$ has $|S|=\alpha(G)-1$, this ridge degree is the size of the fiber
 \[
 F_G(S)=\{x\in V(G)\setminus S:S\cup\{x\}\in\Omega(G)\}.
 \]
 Equivalently,
 \[
 F_G(S)=V(G_S)=V(\lk_{\Delta(G)}(S)).
 \]
 The set $F_G(S)$ is always a clique of $G$: if two distinct vertices of $F_G(S)$ were non-adjacent, then together with $S$ they would form an independent set larger than $\alpha(G)$.
 
 For a graph $H$ with clique number $r$, and for $Q\in\mathcal K_{r-1}(H)$, the \emph{$(r-1)$-clique-codegree} of $Q$ is
 \[
 \codeg_H(Q)=\bigl|\{x\in V(H)\setminus Q:Q\cup\{x\}\in\mathcal K_r(H)\}\bigr|.
 \]
 The minimum $(r-1)$-clique-codegree of $H$ is
 \[
 \delta_{r-1}^{\cl}(H)=\min\{\codeg_H(Q):Q\in\mathcal K_{r-1}(H)\}.
 \]
 For $r=1$, we use the convention that $\mathcal K_0(H)=\{\emptyset\}$ and $\delta_0^{\cl}(H)=|V(H)|$.
 
 For a well-covered graph $G$, the \emph{$\mathbf W$-index} is
 \[
 \widx(G)=\max\{p\ge1:G\in\mathbf W_p\}.
 \]
 
 We shall use the following localization theorem for $\mathbf W_p$ graphs.
 
 \begin{theorem} \label{thm:HLM-local}
 Let $p\ge 1$ and let $G$ be a graph with $\alpha(G)\ge 2$.  Then
 \[
 G\in \mathbf W_p
 \quad\Longleftrightarrow\quad
 G_x\in \mathbf W_p\text{ and }\alpha(G_x)=\alpha(G)-1\text{ for every }x\in V(G).
 \]
 Moreover, if $G\in \mathbf W_p$ and $S\in\Omega^*(G)$ satisfies $|S|<\alpha(G)$, then $G_S\in\mathbf W_p$ and $\alpha(G_S)=\alpha(G)-|S|$.
 \end{theorem}
 
 The first assertion is Theorem 3.2 of \cite{HoangLevitMandrescu2025}.  The localization assertion is Lemma 2.9 of \cite{HoangLevitMandrescu2025}, together with the standard well-covered localization identity of Plummer and Finbow--Hartnell--Nowakowski \cite{Plummer1993,FinbowHartnellNowakowski1993}.  For $p=2$, related forms were obtained in \cite{Pinter1995,LevitMandrescu2019}.
 
 \section{The \texorpdfstring{$\mathbf W_p$}{Wp} property as ridge thickness}
 
 The following theorem is the structural form of the $\mathbf W_p$ property used throughout the paper.
 
 \begin{theorem}[Ridge-thickness characterization of \texorpdfstring{$\mathbf W_p$}{Wp}]\label{thm:Wp-ridges}
 Let $p\ge 1$, let $G$ be a graph with $\alpha(G)=r\ge 1$, and let $\Delta=\Delta(G)$.  Then the following   are equivalent: 
 \begin{enumerate}
 \item[(a)]  $G\in\mathbf W_p$ 
 \item[(b)] $\Delta$ is pure and every ridge of $\Delta$ has degree at least $p$.
 \item[(c)]
  $G$   is well-covered and   $|F_G(S)|\ge p$
  for every  $S\in\Omega^*(G)$ with $|S|=r-1$.\end{enumerate} 
 \end{theorem}
 
 \begin{proof}
 (b) $\Longleftrightarrow$ (c)  is immediate from the definitions. 
 \vskip0.5em 
 (a) $\Longrightarrow$ (c):      Assume first that $G\in\mathbf W_p$.  Then $G\in\mathbf W_1$, so $G$ is well-covered and $\Delta(G)$ is pure.  Let $S$ be a ridge.  If $r=1$, then $S=\emptyset$, and $|F_G(S)|=|V(G)|\ge p$ by the definition of $\mathbf W_p$.  If $r\ge2$, then $G_S\in\mathbf W_p$ by \Cref{thm:HLM-local}, and $\alpha(G_S)=1$.  Therefore $G_S$ is a complete graph on at least $p$ vertices.  Since $V(G_S)=F_G(S)$, we get $|F_G(S)|\ge p$.
 
 \vskip0.5em 
 (b) $\Longrightarrow$ (a):   Assume that $\Delta(G)$ is pure of dimension $r-1$ and that every ridge has degree at least $p$.  We prove $G\in\mathbf W_p$ by induction on $r$.  If $r=1$, then $G$ is complete and $|V(G)|=\deg_{\Delta}(\emptyset)\ge p$, so $G\in\mathbf W_p$.
 
 Let $r\ge2$.  Fix $x\in V(G)$.  Since $\Delta(G)$ is pure, every maximal independent set of $G_x$ has size $r-1$. Indeed, if $I$ is a maximal independent set  in $G_x$, then $I\cup\{x\}$ is also  a maximal independent set  in $G$ and hence has size $r$.  Hence $G_x$ is well-covered and $\alpha(G_x)=r-1$.  Now let $T$ be a ridge of $\Delta(G_x)$, so $|T|=r-2$.  Then $T\cup\{x\}$ is a ridge of $\Delta(G)$, and
 \[
 F_{G_x}(T)=F_G(T\cup\{x\}).
 \]
 Hence every ridge of $\Delta(G_x)$ has degree at least $p$.  By the induction hypothesis, $G_x\in\mathbf W_p$ for every $x\in V(G)$.  Applying \Cref{thm:HLM-local}, we obtain $G\in\mathbf W_p$.
 \end{proof}
 
 Notice that \Cref{thm:Wp-ridges} does not assume \(\alpha\)-criticality.  It characterizes membership in \(\mathbf W_p\) for arbitrary graphs by purity and ridge thickness; in particular, for well-covered graphs, \(\mathbf W_p\)-membership is controlled entirely by the codimension-one fibers.
 
 This gives an exact numerical invariant.
 
 \begin{corollary}[Exact $\mathbf W$-index]\label{cor:exact-index}
 For a well-covered graph $G$ with $r=\alpha(G)$,
 \[
 \widx(G)=\min_{\substack{S\in\Omega^*(G)\\ |S|=r-1}} |F_G(S)|.
 \]
 Equivalently, if $H=\overline G$, then
 \[
 \widx(G)=\delta_{r-1}^{\cl}(H).
 \]
 \end{corollary}
 
 \begin{proof}
 The first formula is immediate from \Cref{thm:Wp-ridges}.  The second follows because independent sets of $G$ are cliques of $H$, and $F_G(S)$ is exactly the set of vertices extending the corresponding $(r-1)$-clique to an $r$-clique of $H$.
 \end{proof}
 
 \section{\texorpdfstring{$\alpha$}{Alpha}-criticality as generation by ridge links}
 
 Let $G$ be a graph with $\alpha(G)=r$.  For an independent set $S$ of size $r-1$, the fiber $F_G(S)$ is always a clique of $G$: if $x,y\in F_G(S)$ were non-adjacent, then $S\cup\{x,y\}$ would be independent of size $r+1$.
 
 The usual localization test for an edge $xy$ to be critical is $\alpha(G_{xy})=\alpha(G)-1$; variants of this criterion appear in the literature on $\alpha$-critical graphs and edge ideals \cite{HoangLevitMandrescu2025,JaramilloVillarreal2021}.  The next proposition rewrites the same phenomenon in ridge-fiber language.
 
 \begin{proposition}[Fiber cover criterion for \texorpdfstring{$\alpha$}{alpha}-criticality]\label{prop:fiber-critical}
 Let $G$ be a graph with $\alpha(G)=r$.  Then $G$ is $\alpha$-critical if and only if
 \[
 E(G)=\bigcup_{\substack{S\in\Omega^*(G)\\ |S|=r-1}} \binom{F_G(S)}{2}.
 \]
 \end{proposition}
 
 \begin{proof} For every $S\in \Omega^*(G)$ with  $|S|=r-1$, the set  $F_G(S)$ is a clique. Hence the right-hand side is contained in $E(G)$.
 
 Let $xy\in E(G)$ be arbitrary.  If $xy$ is critical, then $G-xy$ has an independent set $M$ of size $r+1$.  The only new independent sets created by deleting $xy$ are those containing both $x$ and $y$, so $x,y\in M$.  Put $S=M\setminus\{x,y\}$.  Then $S$ is independent in $G$, $|S|=r-1$, and both $S\cup\{x\}$ and $S\cup\{y\}$ are maximum independent sets of $G$.  Thus $x,y\in F_G(S)$.
 
 Conversely, if $x,y\in F_G(S)$ for  some  independent set  $S$ of size $r-1$, then $xy\in E(G)$ and the set $S\cup\{x,y\}$ is independent in $G-xy$ of size $r+1$.  Hence $xy$ is critical.
 \end{proof}
 
 In the independence complex $\Delta(G)$, the pair $\{x,y\}$ is a missing edge precisely when $xy\in E(G)$. Therefore Proposition~\ref{prop:fiber-critical} says that every missing edge of the flag complex $\Delta(G)$ is contained in the vertex set of the link of some ridge.
 
 \section{The saturation characterization}
 
 We now prove the full characterization.  The result is stated in graph language, independence-complex language, and complement language.
 
 \begin{theorem}\label{thm:main}
 Let $p\ge1$, let $G$ be a graph, let $r=\alpha(G)$, let $\Delta=\Delta(G)$, and let $H=\overline G$.  The following conditions are equivalent.
 \begin{enumerate}
 \item[(a)] $G$ is $\alpha$-critical and $G\in\mathbf W_p$.
 \item[(b)] $\Delta$ is pure, every ridge of $\Delta$ has degree at least $p$, and every missing edge of $\Delta$ is contained in the vertex set of the link of some ridge.
 \item[(c)] $G$ is well-covered, $|F_G(S)|\ge p$ for every $S\in\Omega^*(G)$ with $|S|=r-1$, and
 \[
 E(G)=\bigcup_{\substack{S\in\Omega^*(G)\\ |S|=r-1}} \binom{F_G(S)}{2}.
 \]
 \item[(d)] $H$ is $K_{r+1}$-saturated, every maximal clique of $H$ has size $r$, and
 \[
 \delta_{r-1}^{\cl}(H)\ge p.
 \]
 \end{enumerate}
 \end{theorem}
 
 \begin{proof}
 The equivalence of (a), (b), and (c) follows by combining \Cref{thm:Wp-ridges} with Proposition~\ref{prop:fiber-critical}.
 
 It remains to compare (c) and (d).  The case $r=1$ is covered by the convention $\mathcal K_0(H)=\{\emptyset\}$: then $G$ is complete, $H$ is edgeless, and the stated conditions reduce to $|V(G)|\ge p$.  Hence we may also read the following argument uniformly for all $r\ge1$.

Since independent sets of $G$ are cliques of $H$, the graph $G$ is well-covered with $\alpha(G)=r$ if and only if every maximal clique of $H$ has size $r$. This  implies that   $H$ is $K_{r+1}$-free.   Moreover, for a clique $Q\in\mathcal K_{r-1}(H)$ corresponding to an independent set $S$  of $G$ with $|S|=r-1$, the fiber $F_G(S)$ is the set of vertices extending $Q$ to a copy of $K_r$ in $H$.  Thus the condition $|F_G(S)|\ge p$ for all $S$ is equivalent to $\delta_{r-1}^{\cl}(H)\ge p$.
 
 Finally, an edge $xy$ of $G$ is a non-edge of $H$.  By Proposition~\ref{prop:fiber-critical}, $xy$ is an $\alpha$-critical edge of $G$ if and only if there exists an $(r-1)$-clique $Q$ of $H$ such that $Q\cup\{x\}$ and $Q\cup\{y\}$ are both $r$-cliques of $H$.  This is equivalent to saying that adding the missing edge $xy$ to $H$ creates a $K_{r+1}$ on $Q\cup\{x,y\}$.  Since $\omega(H)=r$, this is precisely $K_{r+1}$-saturation.
 \end{proof}
 
 \begin{corollary}[Complement correspondence]\label{cor:complement-correspondence}
 Fix integers $p\ge1$ and $r\ge1$.  The complement map $G\mapsto\overline G$ gives a bijection between graphs $G$ with $\alpha(G)=r$ which are $\alpha$-critical and belong to $\mathbf W_p$, and graphs $H$ satisfying the following three conditions:
 \begin{enumerate}[(a)]
 \item $H$ is $K_{r+1}$-saturated;
 \item every maximal clique of $H$ has size $r$;
 \item $\delta_{r-1}^{\cl}(H)\ge p$.
 \end{enumerate}
 Under this correspondence, a ridge $S$ of $\Delta(G)$ is an $(r-1)$-clique $Q$ of $H$, the fiber $F_G(S)$ is precisely the set of vertices extending $Q$ to an $r$-clique of $H$, and $\widx(G)=\delta_{r-1}^{\cl}(H)$.
 \end{corollary}
 
 \begin{proof}
 The bijection is \Cref{thm:main}, stated with $H=\overline G$.  The description of fibers is the identity
 \[
 F_G(S)=\{x\in V(H)\setminus S:S\cup\{x\}\in\mathcal K_r(H)\},
 \]
 for the corresponding clique $S$ of $H$.  The formula for $\widx(G)$ is Corollary~\ref{cor:exact-index}.
 \end{proof}
 
 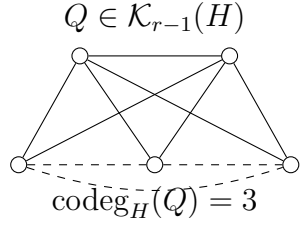
\begin{figure}[H]
 \centering
 \begin{tikzpicture}[scale=0.9, every node/.style={circle, draw, fill=white, inner sep=1.5pt, minimum size=6pt}]
 \node (q1) at (-1.1,1.6) {};
 \node (q2) at (1.1,1.6) {};
 \node (f1) at (-2.0,0) {};
 \node (f2) at (0,0) {};
 \node (f3) at (2.0,0) {};
 \draw (q1)--(q2);
 \foreach \q in {q1,q2}{
  \foreach \f in {f1,f2,f3}{\draw (\q)--(\f);}
 }
 \draw[dashed] (f1)--(f2);
 \draw[dashed] (f2)--(f3);
 \draw[dashed] (f1) to[bend right=18] (f3);
 \node[draw=none, fill=none, rectangle] at (0,2.15) {$Q\in\mathcal K_{r-1}(H)$};
 \node[draw=none, fill=none, rectangle] at (0,-0.55) {$\codeg_H(Q)=3$};
 \end{tikzpicture}
 \caption{A ridge in the complement for the case $r=3$.  The solid upper edge is a $K_{r-1}$, the three lower vertices are its extensions to $K_r$'s, and the dashed lower pairs are missing edges of $H$.  In $G=\overline H$, those dashed pairs become edges inside the fiber clique $F_G(S)$.}
 \label{fig:ridge-fiber}
 \end{figure}

 \section{Consequences and examples} 
 \subsection{Small independence number}
 A graph $G$ is called \emph{maximal triangle-free} if $G$ contains no triangle and, for every non-edge  $e$ of $G$, the graph $G+e$ contains a triangle. 
 For $r=2$, the complement formulation becomes especially sharp.
 
 \begin{corollary}[The case \texorpdfstring{$\alpha(G)=2$}{alpha(G)=2}]\label{cor:alpha2}
 Let $G$ be a graph with $\alpha(G)=2$, and set $H=\overline G$.  Then
 \[
 G\text{ is }\alpha\text{-critical and }G\in\mathbf W_p
 \]
 if and only if $H$ is maximal triangle-free and $\delta(H)\ge p$.  In this case,
 \[
 \widx(G)=\delta(H).
 \]
 \end{corollary}
 
 \begin{proof}
 Apply Theorem \ref{thm:main}(d) with   $r=2$. Observe that   a $K_{3}$-saturated graph is precisely  a maximal triangle-free graph, while  the $1$-clique-codegree coincides  with  ordinary vertex degree.
 \end{proof}
 
 For $r=3$, the condition is an edge-triangle condition.
 
 \begin{corollary}[The case \texorpdfstring{$\alpha(G)=3$}{alpha(G)=3}]\label{cor:alpha3}
 Let $G$ be a graph with $\alpha(G)=3$, and set $H=\overline G$.  Then $G$ is $\alpha$-critical and belongs to $\mathbf W_p$ if and only if the following three conditions hold: 
 \begin{enumerate}
 \item[(a)] $H$ is $K_4$-saturated;
 \item[(b)] every maximal clique of $H$ is a triangle;
 \item[(c)] every edge of $H$ is contained in at least $p$ triangles.
 \end{enumerate}
 Moreover, $\widx(G)$ is the minimum number of triangles containing an edge of $H$.
 \end{corollary}
 
 \subsection{Dense complement rigidity}
 
 The saturation characterization lets us import extremal structure theorems.  The following uses the theorem of Andr\'{a}sfai, Erd\H{o}s and S\'{o}s: every $n$-vertex $K_{r+1}$-free graph with minimum degree greater than $\frac{3r-4}{3r-1}n$ is $r$-partite \cite{AndrasfaiErdosSos1974}.
 
 \begin{corollary}[Dense complement rigidity]\label{cor:dense}
 Let $G$ be an $n$-vertex $\alpha$-critical graph in $\mathbf W_p$, and put $r=\alpha(G)\ge2$ and $H=\overline G$.  If
 \[
 \delta(H)>\frac{3r-4}{3r-1}n,
 \]
 then $H$ is complete $r$-partite.  Equivalently, $G$ is the disjoint union of $r$ complete graphs.  The $r$ clique sizes of $G$ are all at least $p$.
 \end{corollary}
 
 \begin{proof}
 By \Cref{thm:main}, $H$ is $K_{r+1}$-free.  The Andr\'{a}sfai--Erd\H{o}s--S\'{o}s theorem implies that $H$ is $r$-partite; write $V(H)=P_1\cup\cdots\cup P_r$ for an $r$-partition.  If $u\in P_i$ and $v\in P_j$ with $i\ne j$ were non-adjacent, then adding the edge $uv$ would preserve $r$-partiteness and hence could not create a $K_{r+1}$.  This contradicts $K_{r+1}$-saturation.  Hence all cross-part pairs are edges, so $H$ is complete $r$-partite.  An $(r-1)$-clique of $H$ choosing one vertex from all parts except $P_i$ has clique-codegree $|P_i|$.  Therefore the condition $\delta_{r-1}^{\cl}(H)\ge p$ says exactly that each part has size at least $p$.  Passing to complements gives the assertion about $G$.
 \end{proof}

\subsection{Edge-count and order bounds}

The complement correspondence also gives numerical restrictions.  One comes from the classical complete-graph saturation bound of Erd\H{o}s, Hajnal and Moon; the other comes from the additional $p$-fold ridge-thickness condition.

\begin{corollary}[A $p$-sensitive saturation bound]\label{cor:edge-bound}
Let $G$ be an $n$-vertex $\alpha$-critical graph in $\mathbf W_p$, and put $r=\alpha(G)\ge2$ and $H=\overline G$.  Then
\[
       e(H)\ge
       \max\left\{(r-1)n-\binom{r}{2},\, \left\lceil\frac{np(r-1)}{2}\right\rceil\right\}.
\]
Equivalently,
\[
       e(G)\le \binom{n}{2}-
       \max\left\{(r-1)n-\binom{r}{2},\, \left\lceil\frac{np(r-1)}{2}\right\rceil\right\}.
\]
Moreover,
\[
       \delta(H)\ge p(r-1).
\]
If $p\ge2$, then $H$ has no conical vertex, i.e., no vertex adjacent to every other vertex.  If
\[
       n<\frac{(3r-1)(r-1)}{3r-4}\,p,
\]
then $H$ is complete $r$-partite.  Equivalently, $G$ is the disjoint union of $r$ complete graphs, each of order at least $p$.  In particular, if $n=rp$, then
\[
       G\cong \underbrace{K_p\cup\cdots\cup K_p}_{r\text{ copies}}.
\]
\end{corollary}

\begin{proof}
By Corollary~\ref{cor:complement-correspondence}, the graph $H$ is $K_{r+1}$-saturated, every maximal clique of $H$ has size $r$, and $\delta_{r-1}^{\cl}(H)\ge p$.  If $n=r$, then $H=K_r$; the clique-codegree condition forces $p=1$, and all displayed estimates are immediate.  Hence we may assume $n\ge r+1$.  The theorem of Erd\H{o}s, Hajnal and Moon gives
\[
       e(H)\ge \sat(n,K_{r+1})=(r-1)n-\binom{r}{2}.
\]
It remains to prove the $p$-dependent bound.  We first show that $\delta(H)\ge p(r-1)$.  Fix $v\in V(H)$, and choose an $r$-clique $C$ containing $v$; such a clique exists because every maximal clique of $H$ has size $r$.  For each $u\in C\setminus\{v\}$, set
\[
       X_u=\{x\in V(H)\setminus (C\setminus\{u\}) : (C\setminus\{u\})\cup\{x\}\in\mathcal K_r(H)\}.
\]
Since $C\setminus\{u\}$ is an $(r-1)$-clique, the clique-codegree condition gives $|X_u|\ge p$.  Also $X_u\subseteq N_H(v)$, because $v\in C\setminus\{u\}$.

The sets $X_u$, for $u\in C\setminus\{v\}$, are pairwise disjoint.  Indeed, suppose that $x\in X_u\cap X_{u'}$ for two distinct vertices $u,u'\in C\setminus\{v\}$.  Since $x\in X_u$, we have $x\notin C\setminus\{u\}$; since $x\in X_{u'}$, we have $x\notin C\setminus\{u'\}$.  Thus $x\notin C$.  The same two containments say that $x$ is adjacent to every vertex of $C\setminus\{u\}$ and to every vertex of $C\setminus\{u'\}$, hence to every vertex of $C$.  Then $C\cup\{x\}$ is a $K_{r+1}$, a contradiction.  Hence
\[
       d_H(v)\ge \sum_{u\in C\setminus\{v\}} |X_u|\ge p(r-1).
\]
The handshake lemma gives
\[
       e(H)\ge \left\lceil\frac{np(r-1)}2\right\rceil,
\]
and the displayed bound for $e(G)$ follows by complementing.

Now assume $p\ge2$.  If $c$ were a conical vertex of $H$, then $c$ would lie in an $r$-clique $Q\cup\{c\}$ with $Q\in\mathcal K_{r-1}(H)$.  The vertex $c$ extends $Q$ to an $r$-clique, and no other vertex can extend $Q$, because any second extension would form a $K_{r+1}$ together with $Q\cup\{c\}$.  Thus $\codeg_H(Q)=1$, contradicting $\delta_{r-1}^{\cl}(H)\ge p\ge2$.  Hence $H$ has no conical vertex.

If $n<\frac{(3r-1)(r-1)}{3r-4}p$, then the minimum-degree bound gives
\[
       \delta(H)\ge p(r-1)>\frac{3r-4}{3r-1}n.
\]
The complete $r$-partite conclusion follows from Corollary~\ref{cor:dense}.  Finally, when $n=rp$, the strict inequality above holds because
\[
       \frac{(3r-1)(r-1)}{3r-4}-r=\frac{1}{3r-4}>0.
\]
The complete components of $G$ all have order at least $p$ and total order $rp$, so each has order exactly $p$.
\end{proof}

\subsection{Edge-localization is not necessary}
 
The next example is a sharp form of the obstruction to reversing the sufficient condition in Theorem 4.4 of \cite{HoangLevitMandrescu2025}.  That theorem assumes $G_{ab}\in\mathbf W_{p-1}$ for every edge $ab$.  The smallest case $q=2$ is drawn below; the family shows that this hypothesis is not necessary for any $p\ge2$.
 
\begin{example}[A clique blow-up of a $7$-cycle]\label{ex:doubled-c7}
Let $q\ge2$, and let $G_q$ be the graph obtained from a $7$-cycle by replacing each vertex $i\in\mathbb Z_7$ by a $q$-vertex clique $V_i$, and by making $V_i$ complete to $V_{i-1}$ and $V_{i+1}$, with indices modulo $7$.  Equivalently, $G_q$ is the lexicographic product $C_7[K_q]$.
 
The independent sets of $G_q$ are obtained by choosing at most one vertex from each $V_i$, and the chosen indices must form an independent set of $C_7$.  Since $C_7$ is well-covered with independence number $3$, the graph $G_q$ is well-covered with $\alpha(G_q)=3$.
 
Let $S$ be an independent set of $G_q$ with $|S|=2$.  The two indices used by $S$ form an independent $2$-set in $C_7$, and this set extends to a maximum independent $3$-set of $C_7$.  Each available extending index contributes all $q$ vertices of the corresponding clique $V_i$ to the fiber $F_{G_q}(S)$.  Thus $|F_{G_q}(S)|\ge q$ for every such $S$.  Moreover, equality occurs, for instance, when the two indices of $S$ are $1$ and $4$; the only extending index in $C_7$ is then $6$.  Hence
\[
        w(G_q)=q.
\]
In particular, by \Cref{thm:Wp-ridges}, $G_q\in\mathbf W_q$.
 
The graph $G_q$ is also $\alpha$-critical.  If $ab$ is an internal edge of some $V_i$, then in $G_q-ab$ the vertices $a,b$, together with one vertex from each of $V_{i+2}$ and $V_{i+4}$, form an independent set of size $4$.  If $ab$ joins $V_i$ and $V_{i+1}$, then in $G_q-ab$ the vertices $a,b$, together with one vertex from each of $V_{i+3}$ and $V_{i+5}$, form an independent set of size $4$.  Since $\alpha(G_q)=3$, every edge is $\alpha$-critical.
 
Now choose $a\in V_1$ and $b\in V_2$ with $ab\in E(G_q)$.  The localization $(G_q)_{ab}=G_q-(N_{G_q}(a)\cup N_{G_q}(b))$ is the induced subgraph on $V_4\cup V_5\cup V_6$.  In this graph a single vertex of $V_5$ is a maximal independent set of size $1$, while choosing one vertex from $V_4$ and one vertex from $V_6$ gives an independent set of size $2$.  Therefore $(G_q)_{ab}$ is not well-covered.  In particular, $(G_q)_{ab}\notin\mathbf W_1$, and hence $(G_q)_{ab}\notin\mathbf W_{q-1}$.
\end{example}
 
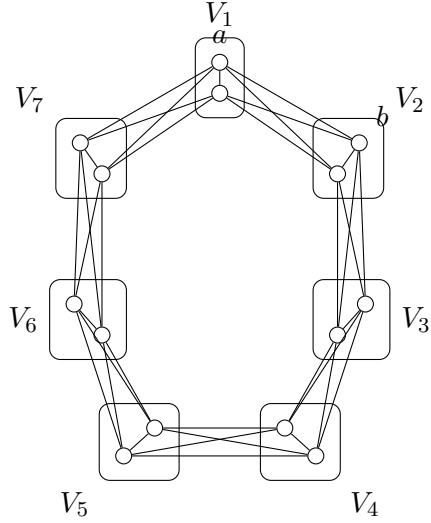
\begin{figure}[H]
\centering
\begin{tikzpicture}[scale=0.82,
  point/.style={circle, draw, fill=white, inner sep=1.2pt, minimum size=6pt},
  block/.style={draw, rounded corners, inner sep=6pt},
  classlabel/.style={draw=none, fill=none, font=\small}
]
\node[point,label={[classlabel]above:$a$}] (v1a) at (0.00,3.25) {};
\node[point] (v1b) at (0.00,2.75) {};
\node[point,label={[classlabel]above right:$b$}] (v2a) at (2.25,1.95) {};
\node[point] (v2b) at (1.90,1.45) {};
\node[point] (v3a) at (2.35,-0.65) {};
\node[point] (v3b) at (1.90,-1.15) {};
\node[point] (v4a) at (1.05,-2.65) {};
\node[point] (v4b) at (1.55,-3.10) {};
\node[point] (v5a) at (-1.05,-2.65) {};
\node[point] (v5b) at (-1.55,-3.10) {};
\node[point] (v6a) at (-2.35,-0.65) {};
\node[point] (v6b) at (-1.90,-1.15) {};
\node[point] (v7a) at (-2.25,1.95) {};
\node[point] (v7b) at (-1.90,1.45) {};

\begin{scope}[on background layer]
\node[block, fit=(v1a)(v1b), label={[classlabel]above:$V_1$}] {};
\node[block, fit=(v2a)(v2b), label={[classlabel]above right:$V_2$}] {};
\node[block, fit=(v3a)(v3b), label={[classlabel]right:$V_3$}] {};
\node[block, fit=(v4a)(v4b), label={[classlabel]below right:$V_4$}] {};
\node[block, fit=(v5a)(v5b), label={[classlabel]below left:$V_5$}] {};
\node[block, fit=(v6a)(v6b), label={[classlabel]left:$V_6$}] {};
\node[block, fit=(v7a)(v7b), label={[classlabel]above left:$V_7$}] {};
\end{scope}

\draw (v1a)--(v1b); \draw (v2a)--(v2b); \draw (v3a)--(v3b); \draw (v4a)--(v4b);
\draw (v5a)--(v5b); \draw (v6a)--(v6b); \draw (v7a)--(v7b);
\foreach \x in {a,b}{\foreach \y in {a,b}{
  \draw (v1\x)--(v2\y); \draw (v2\x)--(v3\y); \draw (v3\x)--(v4\y);
  \draw (v4\x)--(v5\y); \draw (v5\x)--(v6\y); \draw (v6\x)--(v7\y); \draw (v7\x)--(v1\y);
}}
\end{tikzpicture}
\caption{The $q=2$ member of the clique blow-up family $G_q=C_7[K_q]$.  Each $V_i$ is a clique, and consecutive classes are joined completely.  The general case is obtained by replacing each displayed pair by a $q$-vertex clique.  For $a\in V_1$ and $b\in V_2$, the localization $(G_q)_{ab}$ is induced by $V_4\cup V_5\cup V_6$, which is not well-covered.}
\label{fig:doubled-c7}
\end{figure}
 
The Petersen graph gives a complementary instance of the same phenomenon in the especially transparent case $\alpha(G)=2$.
 
\begin{example}[The complement of the Petersen graph]\label{ex:petersen}
Let $P$ be the Petersen graph and let $G=\overline P$.  The Petersen graph is triangle-free, has diameter two, and is $3$-regular; hence it is maximal triangle-free.  By Corollary~\ref{cor:alpha2}, $G$ is $\alpha$-critical, $\alpha(G)=2$, and
\[
\widx(G)=\delta(P)=3.
\]
Thus $G\in\mathbf W_3$ and $G\notin\mathbf W_4$.
 
However, if $ab\in E(G)$, then $a$ and $b$ are non-adjacent in $P$.  In the Petersen graph any two non-adjacent vertices have exactly one common neighbor.  Therefore $G_{ab}\cong K_1$ for every edge $ab$ of $G$.  Since $K_1\notin\mathbf W_2$, the condition $G_{ab}\in\mathbf W_{p-1}$ fails for $p=3$ although $G$ is $\alpha$-critical and belongs to $\mathbf W_3$.
\end{example}
 
\begin{figure}[H]
\centering
\begin{tikzpicture}[scale=1.1, every node/.style={circle, draw, fill=white, inner sep=1.5pt, minimum size=6pt}]
\foreach \i in {0,...,4}{
 \node (o\i) at ({90+72*\i}:2.0) {};
 \node (i\i) at ({90+72*\i}:0.85) {};
}
\foreach \i in {0,...,4}{
 \pgfmathtruncatemacro{\j}{mod(\i+1,5)}
 \draw (o\i)--(o\j);
 \draw (o\i)--(i\i);
}
\draw (i0)--(i2)--(i4)--(i1)--(i3)--(i0);
\end{tikzpicture}
\caption{The Petersen graph $P$.  Its complement is an $\alpha$-critical graph in $\mathbf W_3$ with $\alpha=2$, by the maximal triangle-free characterization.}
\label{fig:petersen}
\end{figure}
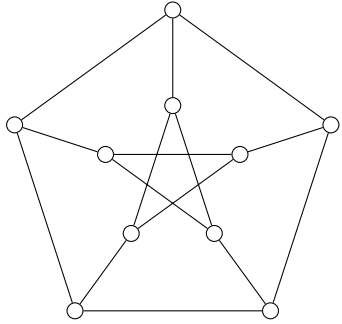

\section{An algebraic corollary and a homological question}
 
Fix a field $k$.  For a simple graph $G$ with vertex set $V(G)=\{x_1,\ldots,x_n\}$, the associated edge ideal of $G$ is the quadratic squarefree monomial ideal $I(G)=(x_ix_j \mid x_ix_j\in E(G))$ in the polynomial ring $R=k[x_1,\ldots, x_n]$.  We say that $G$ is Cohen--Macaulay, respectively Gorenstein, over $k$ if $R/I(G)$ is Cohen--Macaulay, respectively Gorenstein.  For a graph without isolated vertices,
\[
\text{Gorenstein over }k \Longrightarrow \text{Cohen--Macaulay over }k \Longrightarrow \text{well-covered}.
\]
 
Hoang and Trung proved that, for graphs without isolated vertices, the property of being triangle-free and Gorenstein over the field $k$ is equivalent to being triangle-free and belonging to $\mathbf W_2$ \cite{HoangTrung2016}.  Combining their theorem with \Cref{thm:Wp-ridges} gives the following reformulation in ridge and complement language.
 
\begin{corollary}\label{cor:gorenstein-ridge}
Let $k$ be a field.  Let $G$ be a graph without isolated vertices, let $r=\alpha(G)$, and put $H=\overline G$.  The following are equivalent.
\begin{enumerate}[(a)]
 \item $G$ is triangle-free and Gorenstein over $k$.
 \item $G$ is triangle-free and belongs to $\mathbf W_2$.
 \item $G$ is triangle-free, well-covered, and $|F_G(S)|\ge 2$ for every $S\in\Omega^*(G)$ with $|S|=r-1$.
 \item Every maximal clique of $H$ has size $r$, $\alpha(H)= 2$, and $\delta_{r-1}^{\cl}(H)\ge 2$.
\end{enumerate}
\end{corollary}
 
\begin{proof}
The equivalence of (a) and (b) is the triangle-free Gorenstein theorem of Hoang and Trung over the field $k$ \cite{HoangTrung2016}.  The equivalence of (b) and (c) is \Cref{thm:Wp-ridges} with $p=2$.  Finally, $G$ is triangle-free without isolated vertices  if and only if $\alpha(\overline G)=2$, well-coveredness of $G$ is equivalent to all maximal cliques of $H=\overline G$ having size $r$, and the fiber size condition is precisely the clique-codegree condition $\delta_{r-1}^{\cl}(H)\ge2$.
\end{proof}
 
\begin{question}\label{q:triangle-free-CM}
For a fixed field $k$, characterize the triangle-free graphs $G$ for which $R/I(G)$ is Cohen--Macaulay over $k$.
\end{question}
 
The field should be specified in Question~\ref{q:triangle-free-CM}: outside the triangle-free Gorenstein setting, Cohen--Macaulay and Gorenstein behavior of independence complexes may depend on the base field.  A more focused version is to describe the additional homological constraints, beyond purity and the ridge-thickness conditions above, that distinguish the triangle-free Cohen--Macaulay graphs among triangle-free well-covered graphs.

\section{Closing remarks and further directions}
 
The ridge-thickness characterization, \Cref{thm:Wp-ridges}, is independent of $\alpha$-criticality: for every well-covered graph it computes the exact largest $p$ for which the graph lies in $\mathbf W_p$.  The main theorem then adds one further condition, namely generation of all missing edges by ridge links, or equivalently $K_{r+1}$-saturation in the complement.
 
Thus the characterization separates a well-covered extension problem from an edge-criticality problem.  Examples~\ref{ex:doubled-c7} and \ref{ex:petersen} also show that the edge-localization condition $G_{ab}\in\mathbf W_{p-1}$ is a genuinely stronger local hypothesis; it is useful as a sufficient condition, but it is not forced by $\alpha$-criticality together with membership in $\mathbf W_p$.
 
This leaves a sharper local problem between the sufficient condition of \cite{HoangLevitMandrescu2025} and the global characterization above.
 
\begin{question}[Edge-localization refinement]\label{q:edge-localization}
Let $p\ge2$.  Characterize the $\alpha$-critical graphs $G\in\mathbf W_p$ for which
\[
        G_{ab}\in\mathbf W_{p-1}
\]
for every edge $ab\in E(G)$.
\end{question}
 
In the complement language, if $H=\overline G$ and $ab$ is a non-edge of $H$, then
\[
        \overline{G_{ab}}=H[N_H(a)\cap N_H(b)].
\]
Since $H$ is $K_{r+1}$-saturated, this common-neighborhood graph has clique number $r-1$.  Therefore Question~\ref{q:edge-localization} asks which graphs in this complement correspondence have all such common-neighborhood graphs pure with minimum $(r-2)$-clique-codegree at least $p-1$.
 
The complement correspondence in Corollary~\ref{cor:complement-correspondence} is meant to be used in both directions.  Starting from $G$, it turns the independence-extension condition into a clique-codegree lower bound in $\overline G$; starting from a saturated graph $H$, it constructs an $\alpha$-critical graph $\overline H$ and computes its $\mathbf W$-index.  The saturation consequences in Corollaries~\ref{cor:dense} and~\ref{cor:edge-bound} illustrate this transfer: extremal, degree, and order information about $K_{r+1}$-saturated complements becomes structural information about $\alpha$-critical $\mathbf W_p$ graphs.  A natural next saturation problem is to determine the sharp minimum of $e(H)$ among $n$-vertex graphs satisfying the three conditions in Corollary~\ref{cor:complement-correspondence}, as a function of $n,r$, and $p$.

 \section*{Data availability}
 Data sharing is not applicable to this article, as no datasets were generated or analyzed during the current study.
 
 \section*{Declarations}
 \noindent\textbf{Conflict of interest.} The authors declare that they have no conflict of interest.


\begin{thebibliography}{99}
 
 \bibitem{AndrasfaiErdosSos1974}
 B. Andr\'{a}sfai, P. Erd\H{o}s, and V. T. S\'{o}s,
 \newblock On the connection between chromatic number, maximal clique and minimal degree of a graph,
 \newblock \emph{Discrete Mathematics} 8 (1974), no. 3, 205--218.
 
 \bibitem{BeinekeHararyPlummer1967}
 L. W. Beineke, F. Harary, and M. D. Plummer,
 \newblock On the critical lines of a graph,
 \newblock \emph{Pacific Journal of Mathematics} 22 (1967), no. 2, 205--212.
 
 \bibitem{Berge1982}
 C. Berge,
 \newblock Some common properties for regularizable graphs, edge-critical graphs and B-graphs,
 \newblock \emph{Annals of Discrete Mathematics} 12 (1982), 31--44.
 
 \bibitem{BermudoFernau2012}
 S. Bermudo and H. Fernau,
 \newblock Lower bounds on the differential of a graph,
 \newblock \emph{Discrete Mathematics} 312 (2012), 3236--3250.
 
 \bibitem{CastrillonCruzReyes2016}
 I. D. Castrill\'{o}n, R. Cruz, and E. Reyes,
 \newblock On well-covered, vertex decomposable and Cohen--Macaulay graphs,
 \newblock \emph{Electronic Journal of Combinatorics} 23 (2016), no. 2, 17 pp.
 
 \bibitem{ErdosGallai1961}
 P. Erd\H{o}s and T. Gallai,
 \newblock On the minimal number of vertices representing the edges of a graph,
 \newblock \emph{Publications of the Mathematical Institute of the Hungarian Academy of Sciences} 6 (1961), 181--203.
 
 \bibitem{ErdosHajnalMoon1964}
 P. Erd\H{o}s, A. Hajnal, and J. W. Moon,
 \newblock A problem in graph theory,
 \newblock \emph{The American Mathematical Monthly} 71 (1964), no. 10, 1107--1110.
 
 \bibitem{Favaron1982}
 O. Favaron,
 \newblock Very well-covered graphs,
 \newblock \emph{Discrete Mathematics} 42 (1982), 177--187.
 
 \bibitem{FinbowHartnellNowakowski1993}
 A. Finbow, B. Hartnell, and R. Nowakowski,
 \newblock A characterization of well-covered graphs of girth 5 or greater,
 \newblock \emph{Journal of Combinatorial Theory. Series B} 57 (1993), 44--68.
 
 \bibitem{HoangLevitMandrescuPhamSSRN}
 D. T. Hoang, V. E. Levit, E. Mandrescu, and M. H. Pham,
 \newblock On the unimodality of the independence polynomial of clique corona graphs,
 \newblock Available online at SSRN: \url{http://dx.doi.org/10.2139/ssrn.4293649}.
 
 \bibitem{HoangLevitMandrescuPhamLog}
 D. T. Hoang, V. E. Levit, E. Mandrescu, and M. H. Pham,
 \newblock Log-concavity of the independence polynomials of $\mathbf W_p$ graphs, \newblock \emph{Discrete Mathematics}  349 (2026), 115109.
   
 \bibitem{HoangTrung2016}
 D. T. Hoang and T. N. Trung,
 \newblock A characterization of triangle-free Gorenstein graphs and Cohen--Macaulayness of second powers of edge ideals,
 \newblock \emph{Journal of Algebraic Combinatorics} 43 (2016), 325--338.
 
 \bibitem{HoangTrung2018}
 D. T. Hoang and T. N. Trung,
 \newblock Buchsbaumness of the second powers of edge ideals,
 \newblock \emph{Journal of Algebra and Its Applications} 17 (2018), no. 6, 1850117.
 
 \bibitem{HoangLevitMandrescu2025}
 D. T. Hoang, V. E. Levit, and E. Mandrescu,
 \newblock Structural properties and characterizations of $\mathbf W_p$ class,
 \newblock   \emph{Discussiones Mathematicae Graph Theory},  \newblock Available online at   \url{https://doi.org/10.7151/dmgt.2623}.
 
 \bibitem{JaramilloVillarreal2021}
 D. Jaramillo and R. H. Villarreal,
 \newblock The $v$-number of edge ideals,
 \newblock \emph{Journal of Combinatorial Theory. Series A} 177 (2021), 105310.
 
 \bibitem{LevitMandrescu2017}
 V. E. Levit and E. Mandrescu,
 \newblock The Roller-Coaster conjecture revisited,
 \newblock \emph{Graphs and Combinatorics} 33 (2017), 1499--1508.
 
 \bibitem{LevitMandrescu2019}
 V. E. Levit and E. Mandrescu,
 \newblock 1-well-covered graphs revisited,
 \newblock \emph{European Journal of Combinatorics} 80 (2019), 261--272.
 
 \bibitem{Pinter1995}
 M. R. Pinter,
 \newblock A class of planar well-covered graphs with girth four,
 \newblock \emph{Journal of Graph Theory} 19 (1995), 69--81.
 
 \bibitem{Pinter1992}
 M. R. Pinter,
 \newblock Planar regular one-well-covered graphs,
 \newblock \emph{Congressus Numerantium} 91 (1992), 159--159.
 
 \bibitem{Pinter1991}
 M. R. Pinter,
 \newblock $W_2$ graphs and strongly well-covered graphs: two well-covered graph subclasses,
 \newblock Ph.D. thesis, Vanderbilt University, Department of Mathematics, 1991.
 
 \bibitem{Plummer1970}
 M. D. Plummer,
 \newblock Some covering concepts in graphs,
 \newblock \emph{Journal of Combinatorial Theory} 8 (1970), 91--98.
 
 \bibitem{Plummer1993}
 M. D. Plummer,
 \newblock Well-covered graphs: survey,
 \newblock \emph{Quaestiones Mathematicae} 16 (1993), 253--287.
 
 \bibitem{Plummer1967}
 M. D. Plummer,
 \newblock On a family of line-critical graphs,
 \newblock \emph{Monatshefte f\"ur Mathematik} 71 (1967), no. 1, 40--48.
 
 \bibitem{Staples1975}
 J. W. Staples,
 \newblock On some subclasses of well-covered graphs,
 \newblock Ph.D. thesis, Vanderbilt University, 1975.
 
 \bibitem{Staples1979}
 J. W. Staples,
 \newblock On some subclasses of well-covered graphs,
 \newblock \emph{Journal of Graph Theory} 3 (1979), 197--204.
 
 \bibitem{ToppVolkmann1992}
 J. Topp and L. Volkmann,
 \newblock On the well-coveredness of products of graphs,
 \newblock \emph{Ars Combinatoria} 33 (1992), 199--215.
 
 \bibitem{Woodroofe2009}
 R. Woodroofe,
 \newblock Vertex decomposable graphs and obstructions to shellability,
 \newblock \emph{Proceedings of the American Mathematical Society} 137 (2009), 3235--3246.
 
 \end{thebibliography}
\end{document}